\documentclass[12pt]{amsart}
\usepackage{amsmath}
\usepackage{amsthm}
\usepackage{latexsym}
\usepackage{amssymb}

\newtheorem{lemma}{Lemma}

\newcounter{obsctr}

\begin{document}
\def\A {{\mathcal{A}}}
\def\D {{\mathcal{D}}}
\def\R {{\mathbb{R}}}
\def\N {{\mathbb{N}}}
\def\C {{\mathbb{C}}}
\def\Z {{\mathbb{Z}}}
\def\l {\lambda}
\def\tm {\tilde{m}}
\def\ml {\multline}
\def\multiline {\multline}
\def\lessim {\lesssim}
\def\ls {\lesssim}
\def\.ls {\leq {}_.C_0}
\def\a{\alpha}
\def\b{\beta}
\def\phi{\varphi}
\def\L{\Lambda}
\def\e{\epsilon}
\def\epsilon{\varepsilon}
\def\olm{\overline{L_m}}
\def\ol{\overline{L}}
\def\oz{\overline{z}}
\def\be{\begin{equation}}
\def\\[{\begin{equation}}
\def\ee{\end{equation}}
\def\\]{\end{equation}}
\title{On local Gevrey regularity for Gevrey vectors of subelliptic sums of squares - an elementary proof of a sharp Gevrey Kotake-Narasimhan theorem}
\author{David S. Tartakoff}
\address{Department of Mathematics, University
of Illinois at Chicago, m/c 249, 851 S.
Morgan St., Chicago IL  60607, USA}
\email{dstartakoff@gmail.com}
\date{}

\begin{abstract} 
We study the regularity of Gevrey vectors for H\"ormander operators 
$$ P = \sum_{j=1}^m X_j^2 + X_0 + c$$
where the $X_j$ are real vector fields and $c(x)$ is a smooth function, all in Gevrey class $G^{s}.$ The principal hypothesis is that $P$ satisfies the subelliptic estimate: for some $\e >0, \; \exists \,C$ such that
$$\|v\|_{\e}^2 \leq C\left(|(Pv, v)| + \|v\|_0^2\right) \qquad \forall v\in C_0^\infty.$$

We prove directly (without the now familiar use of adding a variable $t$ and proving suitable hypoellipticity for $Q=-D_t^2-P$ and then, using the hypothesis on the iterates of $P$ on $u,$ constructiong a homogeneous solution $U$ for $Q$ whose trace on $t=0$ is just $u$) that for $s\geq 1,$\,$G^s(P,\Omega_0) \subset G^{s/\e}(\Omega_0);$ that is, $$\forall K\Subset \Omega_0, \;\exists C_K: \|P^j u\|_{L^2(K)}\leq C_K^{j+1} (2j)!^s, \;\forall j $$
 $$\implies \forall K'\Subset \Omega_0, \;\exists \tilde C_{K'}:\,\|D^\ell u\|_{L^2(K')} \leq \tilde C_{K'}^{\ell+1} \ell!^{s/\epsilon}, \;\forall \ell.$$
In other words, Gevrey growth of derivatives of $u$ as measured by iterates of $P$ yields Gevrey regularity for $u$ in a larger Gevrey class. 
 
When $\epsilon =1,$ $P$ is elliptic and so we recover the original Kotake-Narasimhan theorem (\cite{KN1962}), which has been studied in many other classes, including ultradistributions (\cite{BJ}).

We are indebted to M. Derridj for multiple conversations over the years.

\end{abstract}
\maketitle
\pagestyle{myheadings}
\markboth{David S. Tartakoff}
{A sharp Kotake-Narasimhan theorem in Gevrey spaces with an elementary proof}

\section{Background}

 In 1972, Derridj and Zuily \cite{DZ} proved $G^s$ hypoellipticity ($Pu \in G^s \implies u \in G^s$) for 
 $$ P = \sum_{j=1}^m X_j^2 + X_0 + c$$
 satisfying
 \begin{equation}\label{ape}\|v\|_{\e}^2 \leq C\left(|(Pv, v)| + \|v\|_0^2\right) \qquad \forall v\in C_0^\infty\end{equation} 
whenever $s>1/\e=q/p \quad \text{with} \quad p,q \in \mathbb{N}^+$  
and very recently, for $P$ with $G^k$ coeffients, $k\in \mathbb{N^+},$ by studying Gevrey vectors for such operators (see below), Derridj was able to sharpen this result to include $s=1/\e=q/p,$ but still with {\it rational} $\e$ and $G^k$ coefficients, $k\in \mathbb{N^+}$ (announced in \cite{D2017} and proven in \cite{D2016prepr}). 

Consider a linear partial differential operator $P$ of order $2$ with real analytic coefficients. An {\it analytic} vector for $P$ is a distribution $u$ such that $u$ behaves analytically when differentiated by powers of $P$ alone: locally, $\|P^ju\| \leq C^j(2j)!$ that is, not all derivatives of $u$ are assumed to behave as though $u$ were analytic, only those sums occurring together precisely as in $P.$

Similarly a Gevrey-s vector $u$ for $P$ (with $P$ only assumed to have Gevrey-s coefficients now) satisfies (locally) $\|P^ju\| \leq C^j(2j)!^s,$ or more precisely, 
$$\forall K\Subset \Omega_0, \;\exists C_K: \|P^j u\|_{L^2(K)}\leq C_K^{j+1} (2j)!^s, \;\forall j .$$ 

Derridj proved that Gevrey-s vectors for $P$ under (\ref{ape}) belong to $G^{s/\e}$  (for $s>1/\e$ if $\e$ is rational) and, to accomplish this, followed the classical method of adding a variable and showing (local) Gevrey hypoellipticity in $G^{1,s/\e}_{t,x}$ for the operator 
\begin{equation}\label{addt} Q=-D_t^{2}-P. \end{equation} 
This yields the result since the (convergent) homogeneous solution
$$U(t,x) = \sum_{\ell \geq 0} (-1)^{\ell}\frac{t^{2\ell}}{(2\ell)!} P^\ell u(x)$$ 
for $Q$ is just equal to $u(x)$ when $t=0.$

Slightly earlier, N. Braun Rodrigues, G. Chinni, P. D. Cordaro and M. R. Jahnke \cite{RCCJ} had obtained a (global) result on a torus for a restricted subclass of such operators $P.$ 
 
 The methods we use also apply to prove the anisotropic hypoellipticity for (\ref{addt}) even for non-rational $\e.$
 
Note that $G^s$ {\it functions} are always Gevrey-s {\it vectors} for any $P.$

\section{general considerations}
There are two main results of this paper. First, the subellipticity index $\e$ need no longer be rational and secondly, we are able to let $s$ equal $1/\e.$ From a technical point of view, the proof is no harder for Gevrey-k coefficients than for analytic coefficients, so we take the vector fields to have analytic coefficients. 

And from a more personal point of view, in reading Derridj's preprint (\cite{D2017}) we could not find a reason why the result should not follow from the direct lines we have established over many decades and which in fact avoid the need to add a variable and deal with (\ref{addt}), despite the historical significance of that approach which in some sense deals with iterates of $P$ in a less obvious way.

In the elliptic case ($\e=1$ in (\ref{est:epsilon}) just below), we recover the celebrated Kotake-Narasimhan theorem (\cite{KN1962}).

The only hypothesis, aside from Gevrey smoothness of the coefficients of $P$ near $\overline{\Omega_0}$, is that for some real $0<\epsilon<1,$ 
\begin{equation} \label{est:epsilon}\| v\|_{\e}^2 \;(+\sum_1^n \|X_j v\|_{L^2}^2)\leq C\{|(Pv, v)_{L^2}| + \|v\|_{L^2}^2\}, \quad \forall v\in C_0^\infty (\Omega_0)
\end{equation}

\section{smoothness}
From the basic {\em a priori} estimate (\ref{est:epsilon}) and those that will follow from it, we have $u\in C^\infty$: from $Pu \in L^2_{loc}$ it will follow from (\ref{est:epsilon}) that $u\in H^\e_{loc}.$ From our estimate (\ref{i.p.2epsilon}) below (for $\|u\|_{2\e}^2$), it will follow that $Pu \in H^\e_{loc},$ (since $P^2u\in L^2_{loc}$) and hence that $u\in H^{2\e}_{loc},$ and similarly from $P^n u \in L^2_{loc},$  that $P^{n-1} u \in H^\e_{loc},$ \ldots, and finally that $u\in H^{(n+1)\e}_{loc}$ (for all $n,$ and hence $u \in C^\infty$).

We will henceforth assume that $u$ is smooth.

And furthermore, there is no difference in the proof if one assumes that the coefficients of $P$ are real analytic functions and not merely Gevrey functions; thus we will not mention the smoothness of the coefficients again. 
 
\section{estimates}

Unless otherwise specified, norms and inner products are in $L^2.$  We have used 
a fractional power $\Lambda^\mu,$ of the Laplacian defined by
$$ \widehat{\Lambda^{\mu} w} (\xi) = (1+|\xi|^2)^{\mu/2}\hat{w}(\xi).$$

In order to obtain estimates at higher and higher levels, we want to replace $v$ by $\phi(x)\Lambda^{\epsilon}v$ above, with $\phi \in C_0^\infty(K^{0}),$ $\phi \equiv 1$ on $K'$ so that we are inserting suitably supported functions into the norm, and we denote by `$(AB)$' {\em{both}} terms with $AB$ and with $BA$ (i.e., the order of $A$ and $B$ is unspecified), 
so $\|(X_j \phi \L^\e) v\|_{L^2}^2 (=\|(X_j (\phi \L^\e)) v\|_{L^2}^2$ is shorthand for $\|X_j \phi \L^\e v\|_{L^2}^2 + \|\phi \L^\e X_j  v\|_{L^2}^2.$ We also will not need to distinguish the various $\{X_j\}$ or explicitly sum over them:
\be\label {epsilon_est}\|\phi\L^\e v\|_{\epsilon}^2 + \|(X \phi\L^\e) v\|_{L^2}^2 \leq C_0\{ |(P\phi\Lambda^{\epsilon}v, \phi\Lambda^{\epsilon}v)_{L^2}| +\|[X, \phi\L^\e]v\|_{L^2}^2\}_.\ee

Finally, a right hand side with $C_0$ in front will be taken to mean that there may be a uniformly `junk' term on the right, in this case $\|\phi \L^\e v\|_{L^2}^2$ from (\ref{est:epsilon}). The constant $C_0$ may take various, but finitely many, values, independent of $\e.$

Thus, keeping both norms and inner products for the moment,
\be\label{i.p.2epsilon} \|\phi\L^\e v\|_{\epsilon}^2 +\|(X \phi \L^\e) v\|_{L^2}^2\ee 
$$\leq C_0\{|(\phi\Lambda^{\epsilon}Pv, \phi\Lambda^{\epsilon}v)| + |([P, \phi\Lambda^{\epsilon}]v, \phi\Lambda^{\epsilon}v)| +\|[X, \phi\L^\e]v\|_{L^2}^2\}.
$$
To expand the brackets, we denote, generically,
$$[P, \phi \L^\e]= [X^2,\phi \L^\e] = X[X,\phi \L^\e] + [X,\phi \L^\e]X$$
$$=X[X,\phi \L^\e] + \phi' \L^\e X +  \phi[X,\L^\e]X
$$
and 
$$\phi[X,\L^\e]X= X\phi[X,\L^\e] - \phi'[X,\L^\e] + \phi [[X,\L^\e],X]
$$
so that expanding the second term on the right in (\ref{i.p.2epsilon}), integrating by parts and interchanging $\phi$ and $\phi',$
$$([P, \phi\Lambda^{\epsilon}]v, \phi\Lambda^{\epsilon}v)  \sim -([X, \phi\Lambda^{\epsilon}]v, X\phi\Lambda^{\epsilon}v) + (\phi \Lambda^{\epsilon}Xv, \phi'\Lambda^{\epsilon}v)  $$
$$- (\phi[X,\L^\e]v,X \phi \L^\e v)
+ (\phi[X,\L^\e]v, \phi' \L^\e v)
+ (\phi [[X,\L^\e],X]v, \phi \L^\e v).
$$
and so, after the usual weighted Schwarz inequalities, (\ref{i.p.2epsilon}) reads 
$$\|\phi\L^\e v\|_{\epsilon}^2 +\|(X \phi \L^\e) v\|^2_{L^2}\leq C_0\{|(\phi\Lambda^{\epsilon}Pv, \phi\Lambda^{\epsilon}v)|$$
\begin{equation}\label{i.p.2epsilon.1}  (+\|[X, \phi\L^\e]v\|_{L^2}^2)+ \|\phi'\L^\e v\|_{L^2}^2  + \|\phi \L_1^\e v\|_{L^2}^2\} + \|\phi \L_2^\e v\|_{-\e}^2\}\end{equation}
where $\L_1^\e$ stands for $[X, \L^\e]$ and $\L_2^\e$ for $[[X,\L^\e],X]$  pseudodifferential operators of order $\e.$ We have suppressed the term $\|\phi [X, \L^\e]\|_{L^2}^2,$ since $\phi [X, \L^\e] = [X, \phi \L^\e] - X(\phi)\L^\e$ both of which already appear above, and now we could omit the term $\|[X,\phi \L^\e]v\|_{L^2}^2$ since the last two terms contain this, though we will preserve it for now because it is suggestive and helps make sense of the second term on the left. 

Everything at this point is well defined. Things become somewhat more complicated as we seek to obtain estimates for higher derivatives. In the end we shall not write everything down explicitly, but for a while it will be important to keep the reader grounded. 

Some features of (\ref{i.p.2epsilon.1}) are that a gain of $\e$ results in at most one derivative on $\phi,$ and clearly this will be important. It is for this reason that we have retained the inner product with $P$ since when an extra derivative threatens, we are able to exchange the two $\phi$'s on the two sides of the inner product and avoid a second derivative on $\phi$ when we have gained only one $\e$ power of $\L.$ And while $v$ is a test function of compact support, our `solution' $u$ will not have compact support. We will introduce a `largest' localizing function, denoted $\Psi,$ which will reside beside $u$ everywhere but in the end be removable modulo infinitely smoothing brackets with precise bounds since there will be other functions of smaller support, such as $\phi,$ to render $\Psi$ unnecessary. 

\section{Personal heuristics} 

This paper has an unusual formulation.

It has become my conviction over the years that a mathematical paper that contains every symbol, and every derivative of a localizing function explicitly notated becomes unreadable. I personally require more guidance in reading a technical paper to aid me in following the formulas. Perhaps, to paraphrase Frege in \cite{F1950}, anyone who understands the flow of the argument and the justification of the flow well enough probably does not actually need all the detailed calculations. 

I would not go that far. But the challenge of following every bracket and every derivative and writing it down would challenge the stomach of the strongest physique and I prefer to omit that much detail and ask the reader to honor the author's honesty and track record and precision and to let the flow suffice in many places. 

I took this approach in my previous paper,  Analytic Hypoellipticity for a New Class of Sums of Squares of Vector Fields in $\mathcal{R}^3$ \cite{Tartakoff3} and in fact the referee wrote that "I guess the author is trying to explain the ideas in his technical calculations by describing them in words with a minimum of symbols, but the words pile on to the point where one needs to be almost as familiar with the calculations as the author himself for them to make sense. A reader might wonder if the author is trying to pull a fast one by substituting a lot of hand-waving for honest computation Ñ if it werenÕt for some of the subsequent pages where the symbols swamp the words. Can't one strike a better balance?" But I have tried for many years to find a better balance and concluded that in this material, and for this author, the answer is "Sadly, no."

\section{Derivatives in terms of powers of $P$} 

The algorithm we will use to achieve estimates in terms of pure powers of $P$ on $u$ is as follows: as above, although now of order $\b,$ modulo uniform, lower order errors, with $\|(X \phi \L^{\b})v\|_{L^2}^2 \,{\equiv \atop \text{def}}\, \|X \phi \L^{\b}v\|_{L^2}^2 + \|\phi \L^{\b} Xv\|_{L^2}^2$, 
\begin{enumerate}
\item First estimate,  for general $\b$ (and $v\in C_0^\infty (\Omega)),$
$$\|\phi\L^{\beta}v\|_{\e}^2 + \|(X \phi \L^{\b})v\|_{L^2}^2$$
$$ \leq C |(P\phi\L^{\b}v, \phi\L^{\b}v)_{L^2}| \quad (+\; \|[X, \phi \L^\b]v\|_{L^2}^2)$$
\item Then commute $P$ past $\phi \L^{\b}$ until it lands beside $v,$ to obtain \break $(\phi \L^{\b}Pv,\phi\L^{\b}v)_{L^2},$ thus requiring treatment of the bracket \break$([P,\phi\L^{\b}]v,\phi\L^{\b}v)_{L^2}.$
\item Next, expand the second inner product of item (2) by writing $P=X^2$ generically, so with $\phi' = \pm [X, \phi]$
$$[P, \phi\L^{\b}] = \{\phi' X + X\circ \phi' \}\L^\b + 2\phi X[X, \L^{\b}] + \phi [[X,\L^{\b}],X]$$
 and thus, integrating $X$ by parts and/or switching $\phi$ and $\phi',$ and using a weighted Schwarz inequality, uniformly in $\b,$ and modulo a small constant times the   LHS in (1),
$$|([P,\phi\L^{\b}] v,\phi\L^{\b} v)| \sim \|\phi'\L^{\b} v\|_{L^2}^2 + \|\phi \L^{\b}_{1} v\|_{L^2}^2 + \|\phi \L^{\b}_{2} v\|_{-\e}^2
$$
where we recall the notation
$$\L_1^\b = [X, \L^\b] \text{ and } \L_2^\b = [[X, \L^\b], X],$$
both of which are of order $\b$.
\item We gather these steps and freely move $\phi$ past powers of $\L,$ since any bracket (whether applied to $v$ or $Pv$) will introduce one or more derivatives on $\phi$ but also decrease the power of $\L$ by at least the same number (not just by that number times $\e <<1$), a trade that will be acceptable (together with the corresponding remainders) and that we will not write explicitly: 
$$\|\phi\L^{\beta+\e}v\|_{L^2}^2 + \|(X \phi \L^{\b})v\|_{L^2}^2 \sim \|\phi\L^{\beta}v\|_{\e}^2 + \|(X \phi \L^{\b})v\|_{L^2}^2 $$
$$\leq C \|\phi \L^{\b-\e}Pv\|_{L^2}^2 + \|\phi'\L^{\b} v\|_{L^2}^2 + \|\phi \L^{\b}_{1} v\|_{L^2}^2 + \|\phi \L^{\b}_{2} v\|_{-\e}^2.
$$

\item These last two terms are of order $\b$ and will be expanded below in the section {\em{Expanding the Brackets}} below. Looking ahead to (\ref{kcomm4}) below, however, for the moment with $\mu = \b$ and any $r,$
$$\phi \L^{\b}_{1} v = \phi\sum_{\ell=1}^{r-1}\frac{1}{\ell!}a^{(\ell)}{(\L^{\b})}^{(\ell)}Dv+  {{}_1R_{\,r} v}
$$
so that (with Lemma 5.1, for $X_j$ with analytic coefficients)
$$\|\phi \L^{\b}_{1} v\|_{L^2} \leq \sum_{\ell=1}^{r-1}\|\phi\frac{a^{(\ell)}}{\ell!}(\L^{\b})^{(\ell)}Dv\|_{L^2} + \|{{}_1R_{\,r} v}\|_{L^2}
$$
$$\leq \sum_{\ell=1}^{\b-1}C_a^\ell\, \b^\ell\|\phi(\L^{\b-\ell})Dv\|_{L^2} + \|{{}_1R_{\,r} v}\|_{L^2}
$$
and the similar but slightly more complicated expression for
$$\L^{-\e}\phi\L_2^\b v = 
\L^{-\e}\phi [[aD,\L^{\beta}],aD]v = \L^{-\e}\phi [[a,\L^{\beta}]D,aD]v$$
$$=
\L^{-\e}\phi ([a,\L^{\beta}]a'D + [[a,\L^\b],aD]D)v$$
$$=
\L^{-\e}\phi ([a,\L^{\beta}]a'D + [[a,\L^\b],a]D^2+ a[a',\L^\b]D)v$$
$$\sim
\L^{-\e}\phi ([[a,\L^\b],a]D^2+ 2a[a',\L^\b]D)v$$
$$\sim \L^{-\e}\phi\sum_{\ell=1}^{r-1} \sum_{\ell'=1}^{r'-1}\frac{1}{\ell!\,\ell'!}a^{(\ell)}a^{(\ell')}{(\L^{\beta})}^{(\ell+\ell')}D^2v$$
$$+ \L^{-\e}\phi\sum_{\ell=1}^{r-1}\frac{1}{\ell!} a^{(\ell +1)}a{(\L^{\beta})}^{(\ell)}Dv
$$
so that, and bringing the coefficients out of the  norm at the expense of additional brackets, as though it were the $L^2$ norm,
$$\|\L^{-\e}\phi\L_2^\b v\|_{L^2} \leq 
\sum_{\tilde\ell=2}^{r-1}C_a^{\tilde\ell}\b^{\tilde\ell}\|\phi\,\L^{\beta-\tilde\ell}D^2v\|_{-\e} + \sum_{\ell=1}^{r-1}C_a^{\ell}\b^{\ell}\|\phi\,\L^{\beta-\ell}Dv\|_{-\e}
$$
or
$$\|\L^{-\e}\phi\L_2^\b v\|_{L^2} \leq 
\sum_{\ell=0}^{r-1}C_a^{\ell}\b^{\ell}\|\phi\,\L^{\beta-\ell}v\|_{-\e}.
$$
As always with pseudodifferential operators, there will be a sum of terms of lower and lower order as dictated by Leibniz formula for brackets, and remainders.

\item We repeat the above steps by applying the estimate in (4) to the terms on the right in (4) producing $\phi\L^{\b-3\e}P^2v,$ $\phi' \L^{\b-2\e}Pv$ and $\phi''\L^{\b-\e} v$, etc. On the right hand side each of the four terms will lead to a `spray' of additional more terms, about four times as many at each next step. The resulting paradigm may be simplified to read 
$$\|\phi\L^{\beta+\e}v\|_{L^2}^2 \leadsto  \|\phi \L^{\b-\e}Pv\|_{L^2}^2 + \|\phi'\L^{\b} v\|_{L^2}^2$$
$$\leadsto  \|\phi \L^{\b-3\e}P^2v\|_{L^2}^2 + \|\phi' \L^{\b-2\e}Pv\|_{L^2}^2 + \|\phi''\L^{\b-\e} v\|_{L^2}^2,$$
and in general, after $k$ iterations, there will be $C^k$ terms of the form 
$$\|\phi^{(k_1)}\L^{\b+\e-(k_1+2k_2)\e}P^{k_2}v\|_{L^2}^2$$
 with $k=k_1+k_2.$ 
\item Continue each iteration until we get to $\b+\e-(k_1+2k_2)\e\leq 0,$ but at the previous step, $\b+\e-(k_1+2k_2)\e\geq 0,$ i.e., $k_1+2k_2=\lceil \frac{\b+\e}{\e} \rceil,$ so that
$$\|\phi\L^{\beta+\e}v\|_{L^2}^2 \leq C^k \|\phi^{(k_1)}\L^{\b+\e-(k_1+2k_2)\e}P^{k_2}v\|_{L^2}^2$$
where the power of $\L$ in each term is non-positive.
\item It remains to apply all of this to our `solution' $u,$ which is subject to the growth of $P^ku,$ not functions like $v$ which are `test' functions and have compact support:
$$\|P^j u\|_{L^2(K)}\leq C_K^{2j+1} (2j)!^s, \;\forall j \text{ for suitable } C_K.$$ 

But we are free to replace $v$ in this estimate by $\Psi u$ where $\Psi \equiv 1$ near the support of $\phi,$ since any error committed in then bringing $\Psi$ out of the norm will be of order $-\infty.$ Modulo this error, then, 
$$\|\phi\L^{\beta+\e}\Psi u\|_{L^2}^2 \leq C^k \|\phi^{(k_1)}\L^{\b+\e-(k_1+2k_2)\e}P^{k_2}u\|_{L^2(K)}^2$$
Our conclusion is that for any $K'\Subset \Omega_0, \;\exists \,C_{K'}: \|D^m u\|_{L^2(K')} \leq \tilde{C}^{m+1} m!^{s/\e}, \;\forall m.$ Taking $\b+\e = m,$ we have
$$\|D^m u\|_{L^2(K')} \leq \underline{\tilde{C}^{m+1}} \sup_{k_1+2k_2=\lceil \frac{m}{\e} \rceil}\|\phi^{(k_1)}\|_\infty \|P^{k_2}u\|_{L^2(K)},$$
in particular, with $\phi \in G^s,$  
$$\|D^m u\|_{L^2(K')} \leq \tilde{C}^{m+1} \sup_{k_1+2k_2=\lceil \frac{m}{\e} \rceil}k_1!^s\| P^{k_2}u\|_{L^2(K)} $$
$$\leq \tilde{C}^{m+1} {\lceil \frac{m}{\e} \rceil}!^s 
\leq C^m(\frac{m}{\e}+1)!^s\leq C_1^{m/\e}(\frac{m}{\e})!^s$$
\end{enumerate}

\section {Expanding the brackets}
In order to write out the above brackets of the previous section concretely, we use a Taylor expansion of the symbol $\l^\mu(\xi)$: $\forall \mu, r,$ and write, with $f=a$ (a coefficient of one of the $X$'s, which will always be accompanied by $\phi$) or by $f=\phi(x)$ itself,
$$([f,\L^{\mu}]v)^\wedge (\xi) = \int \hat{f}(\xi-\eta)\sum_{\ell=1}^{r-1} \frac{(\xi-\eta)^{\ell}{\l^{\mu}}^{(\ell)}(\eta)}{\ell !}\hat{v}(\eta)d\eta+ \widehat{{}_f R_{\,r}v}(\xi)$$
$$=\sum_{\ell=1}^{r-1} \int \frac{\widehat{f^{(\ell)}}(\xi-\eta)}{\ell !}{\l^{\mu}}^{(\ell)}(\eta)\hat{v}(\eta)d\eta+ \widehat{{}_f R_{\,r}v}(\xi)
$$
where
$$\widehat{{}_f R_{\,r}v}(\xi)=\int \frac{\widehat{f^{(r)}}(\xi-\eta)}{r!}\underbrace{\int_0^1dp \ldots \int_0^1dt}_r {\;\l^{\mu}}^{(r)}(\eta + t...p(\xi-\eta))\hat{v}(\eta)d\eta$$
so that (if we write $({\L^{\mu}})^{(\ell)}$ for the operator with symbol ${(\l^\mu)}^{(\ell)}$), with $f=\phi:$ 
\begin{equation}\|[\phi(x),\L^{\mu}]v\|_{L^2} \leq \sum_{\ell=1}^{r-1}\frac{1}{\ell!}\|\phi^{(\ell)}{(\L^{\mu})}^{(\ell)}v\|_{L^2}+ \| {{}_\phi R_{\,r} v}\|_{L^2}\end{equation}
and, recalling that we write $X=aD,$ with $f=a$ (localized):
\be\label{kcomm4}\|\phi[a,\L^{\mu}]Dv\|_{L^2} \leq \sum_{\ell=1}^{r-1}\frac{1}{\ell!}\|\phi a^{(\ell)}{(\L^{\mu})}^{(\ell)}Dv\|_{L^2}+ \| \phi \,{{}_a R_{\,r} v}\|_{L^2}.
\ee
And for the last term in (4) above, $\|\phi \L^{\b}_{2} v\|_{-\e}^2,$ we write
$$ \L^{-\e}\phi \L^{\mu}_{2} v = \L^{-\e}\phi [[a,\L^{\mu}]D,aD]v$$
$$=
\L^{-\e}\phi ([a,\L^{\mu}]a'D + [[a,\L^\mu],aD]D)v$$
$$=
\L^{-\e}\phi ([a,\L^{\mu}]a'D + [[a,\L^\mu],a]D^2+ a[a',\L^\mu]D)v$$
$$\sim
\L^{-\e}\phi ([[a,\L^\mu],a]D^2+ 2a[a',\L^\mu]D)v$$
$$\sim \L^{-\e}\phi\sum_{\ell=1}^{r-1}\frac{1}{\ell!} \sum_{\ell'=1}^{r'-1}\frac{1}{\ell'!}a^{(\ell)}a^{(\ell')}{(\L^{\mu})}^{(\ell+\ell')}D^2v$$
$$+ \L^{-\e}\phi\sum_{\ell=1}^{r-1}\frac{1}{\ell!} a^{(\ell +1)}a{(\L^{\mu})}^{(\ell)}Dv
$$
\begin{lemma} For any $\mu\geq 0$ and any $\sigma,$ 
$${(\l^{\mu})}^{(k)}(\rho)=
\sum_j\underline{(3\mu)^k}
\left\{
\begin{array}{c}
\l^{\mu -k -2j}(\rho), \quad 0\leq j \leq \frac{k}{2}, \;k {\text{ even }}\\
\rho\,\l^{\mu-k -1 -2j}(\rho), \quad 0\leq j \leq \frac{k-1}{2}, \;k {\text{ odd }}
\end{array}
\right.$$
where underlining the coefficient before the brace indicates the number of terms of the form which follow that are present.
\end{lemma}
\begin{proof}
The simplest proof we have found of this result is to denote by $L$ the expression $(1+\rho^2)^{1/2}$ since the pleasant fact that $L'(\rho) = \rho L^{-1}(\rho)$ seems to make the calculations suggestive and transparent. We omit the details but the precise dependence on $\mu$ and $r$ above are important.

To treat the remainders, we divide up the region of integration as we did in (\cite{Tartakoff1973}) into two parts, the first where $|\xi-\rho|\leq \frac{1}{10}|\rho|,$ and hence the action of $R_{\,r}$ is bounded by the $L^1$ norm of derivatives of the coefficients of total order $r$ times $\|\L^{\mu-r}v\|_{L^2} $ and the region where $|\xi|$ (and hence $|\eta|$) are bounded by a multiple of $|\xi-\eta|$ and so that $|\l^\mu(\xi)-\l^\mu(\eta)| \leq C^{\mu}|\xi-\eta|^\mu$ whence for any $M,$ 
$$|([\L^\mu, a(x)] v)^\wedge(\xi)| = |((\l^\mu \hat{a}) \ast \hat{v})(\xi) - (\hat{a} \ast (\l^\mu\hat{v}))(\xi)|$$
$$=|\l^\mu(\xi)\int \hat{a}(\xi-\eta) \hat{v}(\eta)d\eta - \int \hat{a}(\xi -\eta)\l^\mu(\eta) \hat{v}(\eta)\,d\eta|$$
$$=|\int \hat{a}(\xi-\eta)[\l^\mu(\xi)-\l^\mu(\eta)]\hat{v}(\eta) d\eta|$$
$$\leq C^M|\int \widehat{a^{(M+\mu)}}(\xi-\eta)(1+|\eta|^2)^{-M/2}\hat{v}(\eta) d\eta|.$$

\end{proof}
 
\section{Adding a variable}
Previous proofs concerning Gevrey vectors have often, as in Derridj's paper, proved and then used the Gevrey hypoellipticity of the operator 
$$Q=-\frac{\partial^{2}}{\partial t^{2}} - P$$

The proof that a homogeneous solution for $Q$ satisfies $U\in G^{1,s}_{t,x}$ locally for $s\geq 1/\e$ follows using the above techniques and the evident {\em{a priori}} inequality
$$\|W(t,x)\|_{L^2(t),\e(x)}^2 + \sum \|X_j W(t,x)\|_{L^2(t,x)}^2 + \|W(t,x)\|_{1(t),L^2(x)}^2 $$
$$\leq C\{|(QW,W))|_{L^2} + \|W(t,x)\|_{L^2(t,x)}^2\}$$
for $W$ of small support and smooth since the variables are completely separated. 

Then observing that under our hypothesis on the iterates of $P$ on $u,$ the homogeneous convergent series 
$$U(t,x)=\sum_{\ell\geq 0}(-1)^{\ell} \frac{t^{2\ell}}{(2\ell)!}P^\ell u(x)$$
satisfies the above equation in some interval about $t=0$, and hence, restricted to $t=0$ where it is equal to $u,$ will have the desired regularity in Gevrey class.

Finally, since the variables $t$ and $x$ are totally separated in the problem, localizing functions may be taken as products $\phi_1(t) \phi_2(x)$ with $\phi_1$ of Ehrenpreis type or using nested open sets in $t$ while in $x$ Gevrey localization is familiar (and the fact that coefficients now depend on $t$ as well as $x$ presents no new obstacles, even in brackets with $D_t$ or $\L^\b_2$).

\end{document}